\documentclass{amsart}
\usepackage{amssymb
,amsthm
,amsmath
,amscd
,mathtools,hyperref
}
\usepackage{enumerate}
\usepackage{hyperref}
\usepackage[all]{xy}
\usepackage[top=30truemm,bottom=30truemm,left=25truemm,right=25truemm]{geometry}

\newtheorem{thm}{Theorem}[section]
\newtheorem{lem}[thm]{Lemma}
\newtheorem{prop}[thm]{Proposition}
\newtheorem{coro}[thm]{Corollary}
\newtheorem{conj}[thm]{Conjecture}

\theoremstyle{remark}

\newtheorem{rem}[thm]{Remark}

\numberwithin{equation}{section}

\newcommand{\Sp}{\mathrm{Sp}}
\newcommand{\SL}{\mathrm{SL}}
\newcommand{\GL}{\mathrm{GL}}
\newcommand{\Mp}{\mathrm{Mp}}
\newcommand{\Oo}{\mathrm{O}}
\newcommand{\SO}{\mathrm{SO}}
\newcommand{\PGL}{\mathrm{PGL}}
\newcommand{\GSp}{\mathrm{GSp}}
\newcommand{\GSpin}{\mathrm{GSpin}}
\newcommand{\PGSp}{\mathrm{PGSp}}
\newcommand{\PGSO}{\mathrm{PGSO}}
\newcommand{\PGU}{\mathrm{PGU}}

\newcommand{\Hom}{\mathrm{Hom}}
\newcommand{\Gal}{\mathrm{Gal}}
\newcommand{\Res}{\mathrm{Res}}
\newcommand{\Irr}{\mathrm{Irr}}

\makeatletter
\def\iddots{\mathinner{\mkern1mu\raise\p@
	\hbox{.}\mkern2mu\raise4\p@\hbox{.}\mkern2mu
	\raise7\p@\vbox{\kern7\p@\hbox{.}}\mkern1mu}}
\def\adots{\mathinner{\mkern2mu\raise\p@\hbox{.}
 \mkern2mu\raise4\p@\hbox{.}\mkern1mu
 \raise7\p@\vbox{\kern7\p@\hbox{.}}\mkern1mu}}
\makeatother

\allowdisplaybreaks
\title{ The distinction problem for  metaplectic case }
\author{Hengfei Lu}
\address{Department of Mathematics, Weizmann Institute of Science, 234 Herzl St., P.O. B. 26, Rehovot 7610001, ISRAEL }
\email{hengfei.lu@weizmann.ac.il}
\begin{document}

\begin{abstract}
We use  the  theta lifts between $\Mp_2$ and $PD^\times$   to study the  distinction problems for the pair $(\Mp_2(E), \SL_2(F)),$ where $E$ is a quadratic field extension over a nonarchimedean local field $F$ of characteristic zero and $D$ is a quaternion algebra. With a similar strategy, we give a conjectural formula for the multiplicity of distinction problem related to the pair $(\Mp_{2n}(E),\Sp_{2n}(F) ).$
\end{abstract}

\maketitle
\tableofcontents
\section{Introduction}
The distinction problems have been extensively studied for classical groups such as \cite{flicker1991ondist,gan2010shalika,lapid2012unitary,matringe2009distinction,sakellaridis2012periods,prasad2015arelative}. However, very little is known for the distinction problems for covering groups in the literature. This paper focuses on the distinction problems related to the pair $(\Mp_{2n}(E),\Sp_{2n}(F))$, where $\Mp_{2n}(E)$ is the nontrivial two-fold metaplectic cover of $\Sp_{2n}(E)$ and $E/F$ is a quadratic  extension of nonarchimedean local fields. 

 Let $F$ be a finite field extension of $\mathbb{Q}_p.$ Let $W_F$ be its Weil group and $WD_F$ be the Weil-Deligne group.
Let $E=F[\delta]$ be a quadratic extension of  $F$ with Galois group $\Gal(E/F)=\langle\sigma\rangle$, where $\delta^2\in F^\times.$
Let $G$ be a quasi-split reductive group defined over $F$ with the Langlands dual group $\hat{G}$.
Let $\Irr(G(F))$ denote the set of the smooth irreducible admissible representation of $G(F)$, up to isomorphisms. Given a representation $\pi\in \Irr(G(E))$ and a character $\chi$ of $G(F)$, if $\Hom_{G(F)}(\pi,\chi )\neq0 $, then $\pi$ is said to be $(G(F),\chi)$-distinguished. If $\chi$ is a trivial character, then $\pi$ is called a $G(F)$-distinguished representation. Moreover, Dipendra Prasad \cite[\S16]{prasad2015arelative} has a precise conjecture regarding to the multiplicity
\[\dim\Hom_{G(F)}(\pi,\chi_G ) \]
where $\chi_G$ is a quadratic character defined in \cite[\S10]{prasad2015arelative} depending on the reductive group $G$ and the quadratic field extension $E/F$.

It turns out that the disctinction problems for the pair $(\Mp_{2n}(E),\Sp_{2n}(F))$ are related to the Prasad conjecture for the general spin group $\GSpin_{2n+1}$. 
(See \S\ref{subsect.prasad} for more details.)

 Let $(W_n,\langle-,-\rangle)$ be a $2n$-dimensional symplectic space over $F$ with associated symplectic group $\Sp(W_n)=\Sp_{2n}(F)$.
Set $\Mp(W_n)=\Mp_{2n}(F)$ to be the unique nontrivial two-fold metaplectic cover of $\Sp_{2n}(F)$ with multiplication
\[(g_1,\epsilon_1)(g_2,\epsilon_2)=(g_1g_2,\epsilon_1\epsilon_2 c_{Rao}(g_1,g_2)) \]
where $g_i\in \Sp_{2n}(F),\epsilon_i\in\mu_2$ and $c_{Rao}(g_1,g_2)$ is Rao-cocycle. (See \cite[Theorem I.4.5]{kudla1996notes}. )
\par
Let $W_n\otimes_FE=W_{n,E}$ be the $2n$-dimensional symplectic vector space over $E$ with symplectic group $\Sp_{2n}(E)$. So $W'=\Res_{E/F}(W_{n,E})$ with symplectic form $\frac{1}{2}tr_{E/F}\circ \langle-,-\rangle_E$ is a $4n$-dimensional symplectic space over $F.$
 There is a natural group embedding $$i:\Sp_{2n}(E)=\Sp(W_{n,E})\hookrightarrow \Sp(W')=\Sp_{4n}(F)$$ and the preimage of $ i(\Sp_{2n}(E))$ in $\Mp_{4n}(F)$  is isomorphic to the two-fold metaplectic cover $\Mp_{2n}(E)$ of $\Sp_{2n}(E).$ 
There is a commutative diagram
\[\xymatrix{1\ar[r]&\mu_2\ar@{=}[d]\ar[r]&\Mp_{4n}(F)\ar[r]& \Sp_{4n}(F)\ar[r]&1\\1\ar[r] &\mu_2\ar[r]& \Mp_{2n}(E)\ar[u]^i\ar[r]& \Sp_{2n}(E)\ar[r]\ar[u]^i&1 } \]
and there exists a splitting $\Sp_{2n}(F)\hookrightarrow \Mp_{2n}(E) $ due to the group embedding $\Sp(W_n)\hookrightarrow \Sp(W_{n,E})$ (see \S\ref{sect.splitting}).
\par
Given a genuine representation $\tau$ of $\Mp_{2n}(E)$, i.e. $$\tau(\epsilon \tilde{g})=\epsilon\cdot\tau(\tilde{g})\mbox{  for  }  \epsilon\in \mu_2 \mbox{  and  } \tilde{g}\in \Mp_{2n}(E),$$ with central character $\omega_\tau$ satisfying $\omega_\tau(-1)=1,$ where $-1$ means $(-1,1)\in\Mp_{2n}(E)$,  we will consider the metaplectic distinction problem for the pair $(\Mp_{2n}(E),\Sp_{2n}(F)),$ i.e. to determine the multiplicity
\[\dim \Hom_{\Sp_{2n}(F)}(\tau,\mathbb{C} ). \]
In this paper, we will mainly use theta correspondence to deal with such a kind of distinction problem.
\par
Fix a nontrivial additive character $\psi$ of $F.$   Due to Waldspurger's results \cite{waldspurger1991correspondances}, there is a bijection 
\[\xymatrix{\Irr(\Mp_{2}(F))\ar[r]&\Irr(\PGL_2(F))\sqcup \Irr(PD^\times)\ar[l]  }
 \]
where $D$ is the unique quaternion division algebra over $F$ and $\Irr(\Mp_2(F))$
is the set of irreducible genuine smooth representations of $\Mp_2(F)$.
  Gan-Savin established a bijection for higher dimension in \cite{gan2012metaplectic}.
\begin{thm}
	[Gan-Savin] There is a bijection
	\[\theta_\psi:\Irr(\Mp(W_{n}) )\longrightarrow \Irr(\SO(V_{2n+1}^+))\sqcup \Irr(\SO(V_{2n+1}^-)) ,\]
	where $V_{2n+1}^+$ (respectively $V_{2n+1}^-$) is the split (resp. non-split) quadratic space with trivial discriminant and dimension $2n+1$ over $F.$ This bijection is given by the local theta correspondence for the group $\Mp(W_n)\times\SO(V_{2n+1}^\pm),$ depending on $\psi.$ Moreover, the representation $\theta_{\psi}(\tau)$ is tempered (resp. square-integrable) if and only if $\tau\in \Irr(\Mp(W_n))$ is tempered (resp. square-integrable).\label{1}
\end{thm}
Fix an additive character $\psi\circ \frac{1}{2} tr_{E/F} $ of $E$, still  denoted by $\psi$. 
Suppose that  $\tau\in \Irr(\Mp_{2n}(E))$  associated with an enhanced Langlands parameter $(\phi_{\tau},\eta_\tau)$ where $$\phi_{\tau}=\sum_{i=1}^r m_i\phi_i:WD_E\longrightarrow\Sp_{2n}(\mathbb{C}),$$  $\phi_i$ are distinct irreducible representations   with multiplicity $m_i=m_{\phi_{\tau}}(\phi_i)$ in $\phi_{\tau}$ and $\eta_\tau$ is a character of  the component group $A_{\phi_{\tau}}=C_{\phi_{\tau}}/C^\circ_{\phi_{\tau}}$ where $$C_{\phi_{\tau}}=\{g\in\Sp_{2n}(\mathbb{C}):g^{-1}\phi_\tau(t)g=\phi_\tau(t) \mbox{  for all  } t\in WD_E \}$$ is the centralizer of $\phi_{\tau}$ with connected component $C^\circ_{\phi_{\tau}}$. The component group $A_{\phi_{\tau}}$
is given by 
\[A_{\phi_{\tau} }=\bigoplus_{i=1}^r \mathbb{Z}/2\mathbb{Z} \xi_i\cong(\mathbb{Z}/2\mathbb{Z} )^r. \]
We will use $-1$ to denote the sum $\sum_im_i\xi_i $ in $A_{\phi_{\tau}}$. If $\eta_\tau(-1)=1$ (resp. $-1$), then $\theta_{\psi}(\tau)$ is a nonzero representation of $\SO(V_{2n+1}^+)$ (resp. $\SO(V_{2n+1}^-)$) with Langlands parameter $\phi_{\theta_{\psi}(\tau)}= \phi_{\tau}$.
Fix $\ell\in W_F\setminus W_E$. A Langlands parameter $$\phi_{\tau}:WD_E\longrightarrow\Sp_{2n}(\mathbb{C})=\Sp(M)$$ is called \textbf{conjugate-orthogonal} if there exist a bilinear form $B$ on $M$ such that
\[\begin{cases}
B(\phi_{\tau}(t)m,\phi_{\tau}(\ell t\ell^{-1})m')=B(m,m')\\
B(m,\phi_{\tau}(\ell^2)m')=B(m',m)
\end{cases} \]
for all $m,m'\in M$ and $t\in WD_E$.
\par
Denote $\tau^\delta$ to be the representation obtained via the adjoint action of $g_\delta=\begin{pmatrix}
\delta\\&1
\end{pmatrix}$ in the similitude group $\GSp_{2n}(E)$, i.e. $\tau^\delta$ is given by
\[\tau^\delta\big((y,\epsilon)\big)=\tau\big((g_\delta^{-1}yg_\delta,\epsilon)\big) \]
for $(y,\epsilon)\in\Mp_{2n}(E)$. The enhanced $L$-parameter of $\tau^\delta$ is given in \cite[Theorem 1.5]{gan2012metaplectic}. 

 Let $\epsilon(1/2,\phi_\tau)=\epsilon(1/2,\phi_\tau,\psi )$ be the local root number defined in \cite[\S5]{gan2011symplectic}.
\begin{thm}\cite[Theorem 1.5]{gan2012metaplectic}
	Let $\tau\in \Irr(\Mp_{2n}(E))$ with an enhanced $L$-parameter $(\phi_{\tau},\eta_\tau)$. Then  the enhanced $L$-parameter  of the conjugated representation $\tau^\delta$ is given by
	\[(\phi_{\tau}\otimes\chi_{\delta},\eta' )\]  where 
 $\chi_{\delta}(e)=\langle\delta,e \rangle_E$ for $e\in E^\times$,  $\langle-,-\rangle_E$ is the Hilbert symbol 
and \[\eta'(\xi_i)/\eta_\tau(\xi_i)=\epsilon(1/2,\phi_i)\epsilon(1/2,\phi_i\otimes\chi_{\delta} )\chi_{\delta}(-1)^{\dim\phi_i/2 }\]  for  $\xi_i\in A_{\phi_{\tau}}=A_{\phi_{\tau^\delta} } .$
\end{thm}
 Using the see-saw identity and Mackey Theory, we have the following results:
\begin{thm} Assume that $\tau\in \Irr(\Mp_{2}(E))$  with an $L$-parameter $(\phi_\tau,\eta_\tau).$\label{mainthm}
	\begin{enumerate}[(i)]
		\item 	If $\tau$ is an irreducible square-integrable representation  of $\Mp_2(E),$ then
		\[\dim \Hom_{\SL_2(F) }(\tau,\mathbb{C} )=2\cdot\dim \Hom_{\PGL_2(F)}(\pi,\mathbb{C} ), \]
		where $\pi=\theta_{\psi}(\tau^\delta)$ is a representation of $\PGL_2(E)$. 
		Thus $\tau$ is $\SL_2(F)$-distinguished if and only if the Langlands parameter $\phi_{\tau}\otimes\chi_{\delta}$  is conjugate-orthogonal and $$\eta_\tau(-1)=\epsilon(1/2,\phi_{\tau})\epsilon(1/2,\phi_{\tau}\otimes\chi_{\delta})(-1,\delta)_E.$$
		\item If $\tau =\pi_{\psi}(\mu)$ with $\mu^2\neq|-|^{\pm1}$ is an irreducible principal series representation of $\Mp_2(E),$ then
			$$\dim\Hom_{\SL_2(F)}(\pi_{\psi}(\mu\cdot\chi_{\delta}),\mathbb{C})=\begin{cases}
			2,&\mbox{if }\mu|_{E^1}=\mathbf{1},\\
			1,&\mbox{if }\mu|_{F^\times}=\mathbf{1}\mbox{ and }\mu^2\neq\mathbf{1},\\
			0,& \mbox{other cases,}
			\end{cases}$$
		where $E^1=\{e\in E^\times:e\sigma(e)=1 \}$ and $\pi_{\psi}(\mu)$ is defined in \S $3$.
		\item If $\tau$ is the even or odd Weil representation $\omega_{\psi,\chi_a}^\pm$ of $\Mp_2(E),$
		where $\chi_a(e)=(e,a)_E$ for $e\in E^\times$ is the quadratic character associated to $a\in E^\times/(E^\times)^2,$ then $\Hom_{\SL_2(F)}(\omega_{\psi,\chi_a}^-,\mathbb{C} )=0.$
	 Moreover, 
		$$\dim \Hom_{\SL_2(F)}(\omega_{\psi,\chi_a}^+,\mathbb{C} )=\begin{cases}
 2,&\mbox{if }	
 a\in \delta E^1\setminus (\delta E^1\cap (E^\times)^2);\\
0,&\mbox{otherwise }.
		\end{cases}  $$
	\end{enumerate}
\end{thm}
Let us give a brief introduction to the proof of Theorem \ref{mainthm}. Assume that $\tau^\delta$ corresponds to the representation $\pi$ of $\PGL_2(E)$ under $\theta_{\psi}$. Then the sum below (which is well-known) $$\dim\Hom_{\PGL_2(F)}(\pi,\mathbb{C})+\dim\Hom_{PD^\times}(\pi,\mathbb{C}) $$ equals to the dimension 
$\dim\Hom_{\Mp_{2}(E) }(I(0),\tau^\delta)$, where $I(s)$ is the degenerate principal series of $\Mp_4(F)$.
We will consider the double coset decomposition for $\tilde{P}\backslash\Mp_{4n}(F)/\Mp_{2n}(E)$ in general (see Proposition \ref{casselman}), where $P$ is the Siegel parabolic subgroup of $\Sp_{4n}(F)$ and $\tilde{P}$ is the preimage of $P$ in $\Mp_{4n}(F).$ It turns out that only the open orbit contributes to the dimension
$\dim\Hom_{\Mp_{2n}(E) }(I(0),\tau^\delta) $ if $\tau$ is  tempered.
Because the concrete embedding of the stabilizer $\Sp_{2n}(F)$ of the open orbit (unique) in $P\backslash\Sp_{4n}(F)/\Sp_{2n}(E)$ into $\Sp_{2n}(E)$ is different from the natural embedding $\Sp(W_n)\hookrightarrow\Sp(W_{n,E})=\Sp_{2n}(E),$ differing by an adjoint action of $g_\delta,$ we obtain that 
\[\dim\Hom_{\SL_2(F)}(\tau,\mathbb{C})=\dim\Hom_{\Mp_{2}(E)}(I(0),\tau^\delta)=\dim\Hom_{\PGL_2(F)}(\pi,\mathbb{C})+\dim\Hom_{PD^\times}(\pi,\mathbb{C}) .  \]
We also have an analogue result for the higher dimension. (See \S \ref{sect.proof} and \S\ref{subsect.Mp(4)} for more details.)
\par
Now we briefly describe the contents and the organization of this paper.
In \S $2$, we set up the notation about the  local theta lifts. Then we recall the classification for genuine representations of $\Mp_2$ in \S $3.$ In \S\ref{sect.splitting}, we will focus on the explicit splitting $\Sp_{2n}(F)\hookrightarrow\Mp_{2n}(E)$.
The proof of Theorem \ref{mainthm} will be given in \S \ref{sect.proof}. Then we  use the results of metaplectic disctinction problems to deal with the distinction problem of the classical group $\PGSp_4$ related to the Saito-Kurokawa lifts in \S \ref{sect.application}.
Finally we give a short discussion for the relation between the Prasad conjecture for the group $\GSpin_{2n+1}$ and metaplectic distinction problems in \S \ref{subsect.prasad}.
\subsection*{Acknowledgement}
The author is grateful to  Wee Teck Gan for his guidance and numerous discussions when he was doing his Ph.D. study at National University of Singapore. This project was starting from the conversation among  Shiv Prakash PATEL,   Dipendra PRASAD and the author, when he was participating in the Doctoral School "Introduction to Relative Aspects in Representation Theory, Langlands Functoriality and Automorphic Forms" at   CIRM Luminy in May $2016.$ He would like to thank CIRM for supporting his participation as well. This work was partially supported by an MOE Tier one grant R-146-000-228-114.
The author also thanks the anonymous referees for their helpful comments on earlier versions.
\section{The Local Theta Correspondences }
In this section, we will briefly recall some results about the local theta correspondence, following \cite{kudla1996notes}.  

Let $F$ be a local field of characteristic zero.
Consider the dual pair $\Oo(V)\times \Mp(W).$
For our purpose, we may assume that $\dim V$ is odd . Fix a nontrivial additive character $\psi$ of $F.$
Let $\omega_\psi$ be the Weil representation for $\Oo(V)\times \Mp(W).$ 
If $\pi$ is an irreducible (genuine) representation of $\Oo(V)$ (resp. $\Mp(W)$), the maximal $\pi$-isotypic quotient of the Weil representation $\omega_{\psi}$ has the form 
\[\pi\boxtimes\Theta_\psi(\pi)=\pi\boxtimes\Theta_{V,W,\psi}(\pi)~~ (\mbox{resp. }\Theta_{W,V,\psi}(\pi) \boxtimes\pi) \]
for some smooth genuine representation of $\Mp(W)$ (resp. some smooth representation of $\Oo(V)$). We call $\Theta_\psi(\pi )$ or $\Theta_{V,W,\psi}(\pi)$ (resp. $\Theta_{W,V,\psi}(\pi)$)
the big theta lift of $\pi$. Let $\theta_\psi(\pi)$  or $\theta_{V,W,\psi}(\pi)$ (resp. $\theta_{W,V,\psi}(\pi)$) be the maximal semisimple quotient of $\Theta_\psi(\pi),$ which is called the small theta lift of $\pi.$ 
\begin{thm} \cite{gan2014howe,gan2014proof} One has
	\begin{enumerate}[(i)]
		\item $\theta_\psi(\pi)$ is irreducible whenever $\Theta_\psi(\pi)$ is non-zero.
		\item the map $\pi\mapsto \theta_\psi(\pi)$ is injective on its domain.
	\end{enumerate}
\end{thm}
It is called the Howe duality conjecture which has been proven by Waldspurger \cite{waldspurger1990demonstration} when $p\neq2$.

 \subsection{First occurence indices for pairs of orthogonal Witt towers} Let $W_n$ be the $2n$-dimensional symplectic vector space over $F$ with associated metaplectic group $\Mp(W_n)$ and consider the two towers of orthogonal groups attached to the quadratic spaces with trivial discriminant. More precisely, let $V_3$ (resp. $D^\circ$) be the $3$-dimensional quadratic vector space
 in the $4$-dimensional split (resp. non-split) quaternion algebra over $F$, let $\mathbb{H}$ be the  hyperbolic plane over $F$,
\[V_{2r+1}^+=V_3\oplus \mathbb{H}^{r-1}\quad(\mbox{resp.}\quad V_{2r+1}^-= D^\circ\oplus\mathbb{H}^{r-1}) \]
and denote the orthogonal groups by $\Oo(V_{2r+1}^+)$ (resp. $\Oo(V_{2r+1}^-)$). For an irreducible genuine representation $\pi$ of $\Mp(W_n),$ one may consider the theta lifts $\theta^+_r(\pi)$ and $\theta^-_r(\pi)$ to
$\Oo(V^+_{2r+1})$ and $\Oo(V_{2r+1}^-)$ respectively, with respect to a fixed non-trivial additive character $\psi.$ Set
\[\begin{cases}
r^+(\pi)=\inf\{2r+1:\theta^+_r(\pi)\neq0 \};\\
r^-(\pi)=\inf\{2r+1:\theta^-_r(\pi)\neq0 \}.
\end{cases} \]
Then Kudla and Rallis \cite{kudla2005first}, B. Sun and C. Zhu \cite{sun2012conservation} showed:
\begin{thm}
	[Conservation Relation] For any irreducible representation $\pi$ of $\Mp(W_n),$ we have
	\[r^+(\pi)+r^-(\pi)=4n+4=4+2\dim W_n. \]
\end{thm}
\subsection{See-saw identities}
 Let $(V,q)$ be a quadratic vector space over $E.$  Let $V'=\Res_{E/F}V$ be the  same space $V$ but now thought of as a vector space over $F$ with a quadratic form
\[q'(v)=\frac{1}{2}tr_{E/F}(q(v)). \]
If $W_0$ is a symplectic vector space over $F,$ then $W_0\otimes_F E$
is a symplectic vector space over $E.$ Then we have the following isomorphism of symplectic spaces:
\[\Res_{E/F }[(W_0\otimes_F E )\otimes_E V ]\cong W_0\otimes V':=\mathbf{W} \]
There is a pair 
\[(\Mp(W_0),\Oo(V') )\mbox{  and  }(\Mp(W_0\otimes_F E),\Oo(V)) \]
of dual pairs in the metaplectic group $\Mp(\mathbf{W}).$

A pair $(G,H)$ and $(G',H')$ of dual pairs in the metaplectic group $\Mp(\mathbf{W})$ is called a see-saw pair if $H\subset G'$ and $H'\subset G$.
\[\xymatrix{G\ar@{-}[rd]\ar@{-}[d] &G'\ar@{-}[d]\\ H\ar@{-}[ru] &H' } \]
\begin{lem}
	For a see-saw pair of dual pairs $(G,H)$ and $(G',H')$,  let $\pi$ be a genuine representation of $H$ and $\pi'$ of $H'$. Then we have an isomorphism
	\[\Hom_H(\Theta_\psi(\pi'),\pi )\cong \Hom_{H'}(\Theta_{\psi}(\pi),\pi' ). \]
	\end{lem}
	The proof is similar to the one given by Prasad in \cite[Page 6]{prasad1996some}. So we omit it here.

\section{Representations of metaplectic group $\Mp_2$ }
The whole material in this section comes from \cite[\S2]{gan2011shimura}.
The Weil representation $\omega_\psi$ of $\Mp_2(F)$ which is realized on the Schwartz space $\mathcal{S}(F)$ is reducible, and decomposes as
\[\omega_\psi=\omega_{\psi}^+\oplus\omega_{\psi}^-, \]
where $\omega_{\psi}^+$ is realized on the subspace of even functions
and $\omega_{\psi}^- $ is realized on the subspace of odd functions.
\par
Given $a\in F^\times/(F^\times )^2 ,$ we have 
\[\omega_{\psi_a}=\omega_{\psi,\chi_a}=\omega_{\psi,\chi_a}^+\oplus\omega_{\psi,\chi_a}^-, \]
 where $\chi_a(x)=(x,a)_F$ for any $ x\in F^\times.$
Given a character $\mu$ of the torus $T\cong F^\times,$  one may define 
\[\pi_\psi(\mu)=ind_{\tilde{B} }^{\Mp_2(F)}\mu\cdot\chi_\psi\quad(\mbox{normalized induction}) \]
consisting of  smooth functions $f:\Mp_{2}(F)\rightarrow\mathbb{C}$ such that
\[f((tn,\epsilon)\cdot\tilde{g})=\delta_{B}(t)^{1/2}\mu(t)\chi_\psi((t,\epsilon))\cdot f(\tilde{g})\]
 for $t\in T,(tn,\epsilon)\in\tilde{B}$  and $\tilde{g}\in\Mp_{2}(F)$ , 
where $\chi_\psi((t,\epsilon))=\epsilon\cdot\gamma(t,\psi)^{-1}$ is the genuine character of $\tilde{T}$ associated to the Weil index $\gamma(\psi)$ (see \cite[Page 17]{kudla1996notes}) and $$\chi_\psi((-1,1))=\gamma(-1,\psi)^{-1}=\gamma(\psi)/\gamma(\psi_{-1})$$
where $(-1,1)$ lies in the center of $\Mp_2(F)$ and $\psi_{-1}(x)=\psi(-x)$ for $x\in F.$
\par
\begin{prop}[Waldspurger] Assume that $\pi=\pi_\psi(\mu)$ is a principal series of $\Mp_2(F).$
	\begin{enumerate}[(i)]
		\item The representation $\pi_\psi(\mu)$ is
		irreducible if and only if $\mu^2=|-|^{\pm1},$ in which case
		$\pi_{\psi}(\mu)\cong\pi_{\psi}(\mu^{-1}).$
		\item  If $\mu=\chi|-|^{1/2},$ where $\chi$ is a quadratic character, then we have a short exact sequence: 
		\[\xymatrix{0\ar[r]& st_\psi(\chi)\ar[r]&\pi_\psi(\mu)\ar[r]&\omega_{\psi,\chi}^+\ar[r]&0 .} \]
		We call $st_{\psi}(\chi)$ the Steinberg representation associated to $(\psi,\chi).$ When the character $\chi=\mathbf{1},$ we shall simply
		write $st_\psi.$
		\item If $\mu=\chi|-|^{-1/2}$ then we have a short exact sequence
		\[\xymatrix{0\ar[r]& \omega_{\psi,\chi}^+\ar[r]&\pi_\psi(\mu)\ar[r]&sp_\psi(\chi)\ar[r]&0 .} \]
	\end{enumerate}
\end{prop}
\begin{rem}
	The odd Weil representations $\omega_{\psi,\chi}^-$ are supercuspidal.
\end{rem}
There are  explicit theta correspondences under $\psi$ between $\Irr(\PGL_2(F))\sqcup \Irr(PD^\times)$ and $\Irr(\Mp_2(F))$ obtained by Waldspurger in \cite{waldspurger1991correspondances}, which are also summarized in \cite[Page 9]{gan2011shimura}.

\begin{table}[h]
		\renewcommand*{\arraystretch}{1.5}
		\caption{Theta correspondences between $\PGL_2(F)$ and $\Mp_2(F)$}
\begin{tabular}{|c|c|c|c|c|c|}
	\hline
	$\pi \in$ Irr($\PGL_2(F)$)&$\pi(\mu,\mu^{-1} )$&$st_\chi,\chi\neq\mathbf{1}$& st&$\chi\circ\det$&\mbox{supercuspidal}\\
	\hline
	$\tau\in$ Irr($\Mp_2(F)$ )&$\pi_\psi(\mu)$& $st_{\psi,\chi }$&$\omega_\psi^-$&$\omega_{\psi,\chi}^+$&\mbox{supercuspidal}\\
	\hline
\end{tabular}
\end{table}

The representations on the first row correspond to those on the second row under theta correspondence.

\begin{table}[h]
		\renewcommand*{\arraystretch}{1.5}
		\caption{Theta correspondence between $PD^\times$ and $\Mp_2(F)$}
\begin{tabular}{|c|c|c|c| }
	\hline
	$\pi\in $ Irr$(PD^\times)$& $\chi\circ N_{D},\chi\neq\mathbf{1}$&$\mathbf{1}$&$\dim>1$\\
	\hline
	$\tau\in $ Irr $(\Mp_2(F))$&$\omega_{\psi,\chi}^- $&$st_\psi $&supercuspidal\\
	\hline
\end{tabular}
\end{table}

These tables will be very useful in the proof of Theorem \ref{mainthm}.

\section{The splitting $\Sp_{2n}(F)\hookrightarrow \Mp_{2n}(E)$}\label{sect.splitting}
This section focuses on the concrete splitting map $\Sp_{2n}(F)\hookrightarrow\Mp_{2n}(E)$.
Recall that
\[(g_1,\epsilon_1)(g_2,\epsilon_2)=(g_1g_2,\epsilon_1\epsilon_2\cdot c_{Rao}(g_1,g_2)) \]
for $(g_i,\epsilon_i)\in\Mp_{2n}(E)$, where  $$c_{Rao}(-,-):\Sp_{2n}(E)\times\Sp_{2n}(E)\longrightarrow\{\pm1\}$$
is a cocycle defined in \cite[Theorem I.4.5]{Kudla1992}, i.e.
\[c_{Rao}(g_1,g_2)=c\cdot\langle x(g_1),x(g_2)\rangle_E\cdot\langle -x(g_1)x(g_2),x(g_1g_2)\rangle_E \]
where $c\in\{\pm1\}$ is a constant and $x:\Sp_{2n}(E)\longrightarrow E^\times/(E^\times)^2$ is a function defined by Rao. (See \cite[Page 19]{Kudla1992} for more details.) 
 Note that the Hilbert symbol
\[\langle-,-\rangle_E:E^\times\times E^\times\longrightarrow\{\pm1 \} \]
is trivial when restricted on $F^\times\times F^\times$. Then  the restricted cocycle $c_{Rao}(-,-)|_{\Sp_{2n}(F)\times\Sp_{2n}(F) }$
is trivial. Thus there exists a splitting
\[\Sp_{2n}(F)\hookrightarrow\Mp_{2n}(E) \]
due to the group embedding $\Sp_{2n}(F)=\Sp(W_n)\hookrightarrow\Sp(W_{n,E})=\Sp_{2n}(E)$.

Given a representation $\tau\in \Irr(\Mp_{2n}(E))$ and $g\in\GSp_{2n}(E)$, we define
\[\tau^g(\tilde{y})=\tau^g((y,\epsilon))=\tau((g^{-1}yg,\epsilon)) \]
for $\tilde{y}=(y,\epsilon)\in\Mp_{2n}(E)$.
\begin{lem}
	Given $g\in\GSp_{2n}(F)$ and $\tau\in \Irr (\Mp_{2n}(E))$, we have
	\[\dim\Hom_{\Sp_{2n}(F)}(\tau,\mathbb{C} )=\dim\Hom_{\Sp_{2n}(F)}(\tau^g,\mathbb{C})  \]
\end{lem}
It follows from the fact that $\GSp_{2n}(F)$ normalizes the subgroup $\Sp_{2n}(F)$ in $\Mp_{2n}(E)$.
\section{Proof of  Theorem \ref{mainthm}}\label{sect.proof}

The key idea in the proof of Theorem \ref{mainthm} is to use the see-saw identity to transfer the metaplectic distinction problem to that of the pair $(\PGL_2(E),\PGL_2(F))$, which has been studied in  \cite[Theorem 2.5.2]{lu2016}.
\par
Let $P=M\cdot N$ be the Siegel parabolic subgroup of $\Sp(W_n)$. Then the preimage $\tilde{P}$ of $P$ in $\Mp(W_n)$ is of the form
\[\tilde{P}=\tilde{M}\cdot N \]
where $\tilde{M}=\widetilde{\GL}_n(F)$ is a $2$-fold cover of $\GL_n(F).$ There is a natural genuine character of $\widetilde{\GL}_n(F)$ defined by
\[\chi_\psi:(g,\epsilon)\mapsto \epsilon\cdot\gamma(\det(g),\psi)^{-1} \]
for $g\in\GL_n(F)$ and $\epsilon\in\{\pm1 \}.$
\par
Let $\mathcal{I}(s)$ be the degenerate principal series representation of $\Mp(W_n),$ i.e.
\[\mathcal{I}(s)=ind_{\tilde{P}}^{\Mp(W_n)}(\chi_\psi\cdot|\det|^s)\quad(\mbox{normalized induction}) \]
which consists of  the smooth functions $f:\Mp(W_n)\rightarrow\mathbb{C}$ such that
\[f(\tilde{m}n\tilde{g})=\delta_P(m)^{\frac{1}{2}}\chi_\psi(\tilde{m})|\det (m)|^s\cdot f(\tilde{g})\]
   for  $\tilde{m}\in\tilde{M},n\in N$ and $\tilde{g}\in\Mp(W_n)$. 
\begin{lem}\label{orbitforMp(4)}
	Suppose that $n=2$ and the map $$i:\Mp_{2}(E)\longrightarrow\Mp_4(F)$$ is the embedding due to the geometric embedding $\Sp(W_{1,E})\hookrightarrow\Sp(W_2)$.  There is a decreasing $\Mp_2(E)$-equivariant filtration 
	\[\mathcal{I}(s)|_{\Mp_2(E)}\supset \mathcal{I}_1(s)\supset 0\] 
	of $\mathcal{I}(s)|_{\Mp_2(E)}$  such that  $\mathcal{I}_1(s)\cong ind_{\SL_2(F)}^{\Mp_2(E)}\mathbb{C} $(compact induction)
	and $\mathcal{I}(s)/\mathcal{I}_1(s)\cong ind_{\tilde{B}}^{\Mp_{2}(E)}(\chi_\psi\cdot|\det|_E^{s+1/2})$.
\end{lem}
Note that the double coset decomposition for $\tilde{P}\backslash\Mp_4(F)/\Mp_{2}(E)=P\backslash\Sp_4(F)/\SL_2(E)$
implies that
\begin{equation}\label{Mp(4,F)-decomp}
\Mp_4(F)=\tilde{P}\cdot\Mp_2(E)\sqcup \tilde{P}\cdot\eta_0\cdot\Mp_2(E) 
\end{equation}
where $\eta_0$ represents the open orbit, whose stabilizer subgroup in $\Mp_2(E)$ is isomorphic to $\SL_2(F).$ However, the embedding $\SL_2(F)\hookrightarrow\Mp_{2}(E)$ in $\mathcal{I}_1(s)$ is not induced from the natural geometric embedding map $i:\Sp(W_1)\longrightarrow\Sp(W_{1,E}),$ but induced from the composite map $Ad_{g_\delta}\circ i$ of the conjugation map $Ad_{g_\delta}:\SL_2(E)\longrightarrow\SL_{2}(E)$ and the embedding map $i.$ 
\par
 An irreducible genuine admissible representation $\tau$ is said to occur on the boundary of $\mathcal{I}(s)$ at $s=s_0$ if
\[\Hom_{\Mp_2(E)}(\mathcal{I}(s_0)/\mathcal{I}_1(s_0),\tau )\neq0. \]
Moreover, if $\tau$ does not occur on the boundary of $\mathcal{I}(s_0),$ then the cuspidal supports of $\tau$ and $\mathcal{I}(s_0)/\mathcal{I}_1(s_0)$ are disjoint. Hence $\mathrm{Ext}^1_{\Mp_{2}(E)}(\mathcal{I}(s_0)/\mathcal{I}_1(s_0),\tau)=0$ and so the long exact sequence implies
\[\dim \Hom_{\Mp_{2}(E) }(\mathcal{I}(s_0),\tau)=\dim\Hom_{\Mp_{2}(E)}(\mathcal{I}_1(s_0),\tau )= \dim\Hom_{\SL_2(F)}(\mathbb{C},\tau^\delta). \]
\begin{prop}\label{casselman} Let us define the embedding $i:\Mp_2(E)\longrightarrow\Mp_4(F)$ and $\mathcal{I}(s)$ as above.
	If $\tau$ is a tempered representation of $\Mp_2(E),$ then
	$\tau$ does not occur on the boundary of $\mathcal{I}(0).$
\end{prop}
\begin{proof}This is due to the Casselman criterion for temperedness \cite[Proposition 3.5]{ban2013langlands}.
	
	Let us consider the general case. Given a tempered representation $\tau$ of $\Mp_{2n}(E)$ and the degenerate principal series representation $\mathcal{I}(s)$ of $\Mp_{4n}(F),$ then it turns out that $\tau$ does not  occur on the boundary of $\mathcal{I}(s_0)$ with $s_0\geq0$. Note that there are $(n+1)$ orbits in the double coset decomposition $\tilde{P}\backslash\Mp_{4n}(F)/\Mp_{2n}(E),$ where $P$ is the Siegel parabolic subgroup of $\Sp_{4n}(F)$ and $\tilde{P}$ is its preimage in $\Mp_{4n}(F).$ There is a decreasing $\Mp_{2n}(E)$-equivariant filtration 
	\[\mathcal{I}(s)=\mathcal{I}_{n+1}(s)\supset \mathcal{I}_n(s)\supset\cdots\supset \mathcal{I}_2(s)\supset\mathcal{I}_1(s)\supset 0 \]
	of $\mathcal{I}(s)|_{\Mp_{2n}(E)}$ 
	such that $\mathcal{I}_1(s)\cong ind_{\Sp_{2n}(F)}^{\Mp_{2n}(E)}\mathbb{C}$ and $$\mathcal{I}_{i+1}(s)/\mathcal{I}_i(s)\cong ind_{Q_i}^{\Mp_{2n}(E)}(\chi_\psi|\det|_E^{s+i/2}\otimes\mathbb{C})$$
	 for $i=1,2,\cdots,n.$ Here $Q_i\cong (\widetilde{\GL_i(E)}\times\Sp_{2n-2i}(F))\cdot (Mat_{2i,2n-2i}(F)\times Sym^i(E)),$ where $Mat_{m,n}(F)$ is the matrix space consisting of all $m\times n$ matrices and $Sym^i(E)$ consists of symmetric matrices in $Mat_{i,i}(E)$. If 
	\[\Hom_{\Mp_{2n}(E) }(ind_{Q_i}^{\Mp_{2n}(E)}(\chi_\psi |\det|_E^{s+i/2}\otimes\mathbb{C}),\tau )\neq0, \]
	then $$\Hom_{\widetilde{\GL_i(E)}}(\chi_\psi|\det|_E^{s+i/2},R_{\overline{\tilde{Q}_i}}(\tau))\neq0, $$ where $\tilde{Q}_i=(\widetilde{\GL_i(E)}\times_{\mu_2}\Mp_{2n-2i}(E) )\cdot(Mat_{i,2n-2i}(E)\times Sym^i(E))$, $\overline{\tilde{Q}_i}$ stands for the parabolic subgroup opposite to $\tilde{Q}_i$
	and $R_{\overline{\tilde{Q}_i}}$ indicates the normalized Jacquet functor with respect to $\overline{\tilde{Q}_i}$.
		Thanks to \cite[Proposition 3.5]{ban2013langlands} that  the center of $\widetilde{\GL_i(E)}$ acts on any irreducible subquotient of $R_{\overline{\tilde{Q}_i}}$ by a character of the form $\chi_\psi\mu|-|^\alpha_E$ with $\mu$ unitary and $\alpha\leq0$, we obtain that the tempered representation $\tau$ does not occur on the boundary of $\mathcal{I}(0),$ i.e. \[\Hom_{\Mp_{2n}(E) }(\mathcal{I}(0),\tau)\neq0
	\]  
	which implies that   \[ \Hom_{\Mp_{2n}(E) }(\mathcal{I}_1(0),\tau)\cong \Hom_{\Sp_{2n}(F)}(\mathbb{C},\tau^\delta)\neq0.\]
	Thus we have finished the proof.
\end{proof}
Now we start to prove Theorem \ref{mainthm}.
\begin{proof}
[Proof of Theorem \ref{mainthm}] Let $\tau$ be a genuine representation of $\Mp(W_{1,E})=\Mp_2(E)$ where $W_{1,E}=W\otimes_FE.$
	\begin{enumerate}[(i)]
		\item Assume that 
	 $\tau^\delta=\theta_{\psi}(\pi)$, where $\pi$ is a square-integrable representation of $\PGL_2(E)$. Then the character $\eta_{\tau^\delta}$ of the component group $A_{\phi_{\tau^\delta}}=A_{\phi_{\tau}}$ is trivial, i.e. $$\eta_\tau(-1)=\epsilon(1/2,\phi_{\tau})\epsilon(1/2,\phi_{\tau}\otimes\chi_{\delta})\langle -1,\delta\rangle_E$$ and  the see-saw identity implies that
	\[\Hom_{\PGL_2(F) }(\pi,\mathbb{C} )\cong \Hom_{\Mp_2(E) }(R^{2,1}(\mathbf{1} ),\tau^\delta )\hookrightarrow \Hom_{\Mp_2(E)}(\mathcal{I}(0),\tau^\delta ), \]
	where $R^{2,1}(\mathbf{1} )$ is the big theta lift to $\Mp_4(F)$ of the trivial representation of $\PGL_2(F)$. 
	By the structure of the degenerate principal series $\mathcal{I}(0)$ of $\Mp_4(F)$ (see \cite[Proposition 7.2]{gan2014formal}), we have
	\[\mathcal{I}(0)=R^{2,1}(\mathbf{1})\oplus R^{3,0}(\mathbf{1}), \]
	where $R^{3,0}(\mathbf{1})$ is the big theta lift to $\Mp_4(F)$ of the trivial representation from the non-split group $PD^\times$. Hence one has
	\begin{equation*}
	\begin{split}
	\Hom_{\Mp_2(E) }(\mathcal{I}(0),\tau^\delta )&=\Hom_{\Mp_2(E)}(R^{2,1}(\mathbf{1})\oplus R^{3,0}(\mathbf{1}) ,\tau^\delta) \\
	&=\Hom_{\Mp_2(E)}(R^{2,1}(\mathbf{1}),\tau^\delta )\oplus\Hom_{\Mp_2(E)}(R^{3,0}(\mathbf{1}),\tau^\delta )\\
&=	\Hom_{\PGL_2(F)}(\pi,\mathbb{C} )\oplus \Hom_{PD^\times}(\pi,\mathbb{C} ).
	\end{split}
	\end{equation*}
	Since $\tau^\delta$ is a square-integrable representation, Proposition \ref{casselman} implies that $\tau^\delta$ does not occur on the boundary of $\mathcal{I}(0)$. So we can obtain the identity
	\begin{equation*}
	\begin{split}
	\dim \Hom_{\SL_2(F)}(\tau,\mathbb{C} )&=\dim \Hom_{\Mp_2(E) }(\mathcal{I}(0),\tau^\delta )\\
	&=2\dim \Hom_{\PGL_2(F) }(\pi,\mathbb{C})\\
	&=\begin{cases}
	2,&\mbox{  if  }\phi_{\pi}\mbox{ is conjugete-orthogonal};\\
	0,&\mbox{  otherwise.  }
	\end{cases}
	\end{split}
	\end{equation*}
	Here we use the fact that $\dim \Hom_{PD^\times}(\pi,\mathbb{C} )=\dim \Hom_{\PGL_2(F)}(\pi,\mathbb{C} ) $ for a square-integrable representation $\pi$ of $\PGL_2(E)$, due to \cite[Theorem C]{Dipendra1992invariant}.  The multiplicity $\dim \Hom_{\PGL_2(F) }(\pi,\mathbb{C})$ is $1$ if and only if the Langlands parameter $\phi_\pi$ is conjugate-orthogonal (see \cite[Theorem 2.5.3]{lu2016}).
	\par
	Let $D_E$ be the division quaternion algebra over $E$ with a reduced norm $N_{D_E}$. Let $V_{D_E}=(D_E,N_{D_E})$ be the $4$-dimensional non-split quadratic space over $E$ with determinant $1$.
	\par
	If $\tau^\delta=\theta_{\psi}(\pi^{D_E} ) ,$ where $\pi^{D_E}$ is the Jacquet-Langlands correspondence representation  of $PD_E^\times$ associated to $\pi$,  then $\theta_{W_{1,E},V_{5,E},\psi }(\tau^\delta)=\theta_{V_{D_E},W_{2,E},\psi}(\pi^{D_E}\boxtimes\mathbb{C} ) $
	as representations of $\PGSp_4(E)$, where $\pi^{D_E}\boxtimes\mathbb{C}$ is an irreducible representation of $$\mathrm{GSO}(V_{D^E})\cong \frac{D_E^\times\times D_E^\times }{\{(t,t^{-1}) , t\in E^\times\}}.$$ (See \cite[Page 219]{schmidt2005saito} for more details.) Here $V_{5,E}=V_5^+\otimes_FE$ and $\SO(V_{5,E})\cong \PGSp_4(E).$
	 Consider the following see-saw diagrams
	\begin{equation}\label{doubleseesaw}
	\xymatrix{ R^{3,2}(\mathbf{1})&\Mp_4(F)\ar@{-}[d]\ar@{-}[rd]& \PGSp_4(E)\ar@{-}[d]\ar@{-}[rd]& \PGSO(V_8^-)\ar@{-}[ld]\ar@{-}[d]&0 \\ \tau^\delta& \Mp_2(E)\ar@{-}[ru]&\PGSp_4(F)& \PGSO(V_{D_E})&\pi^{D_E}\boxtimes\mathbb{C}  } \end{equation}
	where $R^{3,2}(\mathbf{1})$ is the theta lift to $\Mp_4(F)$ of the trivial representation of $\PGSp_4(F)$.
	Thanks to \cite[Theorem 4.2.18(i)]{lu2016}, one can obtain 
	\begin{equation*}
	\begin{split}
	\Hom_{\Mp_2(E) }(R^{3,2}(\mathbf{1}),\tau^\delta )&= \Hom_{\PGSp_4(F)}(\Theta_{W_{1,E},V_{5,E},\psi }(\tau^\delta),\mathbb{C} )\\
	&=\Hom_{\PGSp_4(F)}(\Theta_{V_{D_E},W_{2,E},\psi}(\pi^{D_E}\boxtimes\mathbb{C}),\mathbb{C} )\\
	&=0.
	\end{split}
	\end{equation*}
	Since $R^{3,2}(\mathbf{1})=\mathcal{I}(1)$ as representations of $\Mp_4(F)$ due to \cite[Proposition 7.2]{gan2014formal} and the square-integrable representation $\tau^\delta$ does not occur on the boundary of $\mathcal{I}(1),$ one has 
	\[\dim \Hom_{\SL_2(F) }(\tau,\mathbb{C} )=\dim \Hom_{\Mp_2(E) }(\mathcal{I}(1),\tau^\delta)=\dim \Hom_{\Mp_2(E) }(R^{3,2}(\mathbf{1}),\tau^\delta)=0. \] 
	In fact, if $\tau^\delta=\theta_\psi(\pi^{D_E}),$ then $\eta_{\tau^\delta}$ is a nontrivial character of the component group $A_{\phi_{\tau}}$, i.e., $$\eta_\tau(-1)=-\epsilon(1/2,\phi_{\rho})\epsilon(1/2,\phi_{\rho}\otimes\chi_{\delta})\langle -1,\delta\rangle_E$$ and $\Theta_{\psi}(\tau^\delta)=0$ as a representation of $\PGL_2(E).$
	\item If $\tau=\pi_{\psi}(\mu)$ with $\mu^2\neq|-|_E$ , then 
	\[\Hom_{\Mp_{2}(E)}(\mathcal{I}(0)/\mathcal{I}_1(0),\tau )=\Hom_{\Mp_{2}(E)}(ind_{\tilde{B}(E)}^{\Mp_{2}(E)}\chi_{\psi}|\det|_E^{1/2},\tau )=0  \]
and so	 $\tau$ does not occur on the boundary of $\mathcal{I}(0)$. Thus
	\begin{equation*}
	\begin{split}
	\dim \Hom_{\SL_2(F)}(\pi_{\psi_E}(\mu\cdot\chi_{\delta}),\mathbb{C} )&=\dim\Hom_{\SL_2(F)}((\tau^\delta)^\vee,\mathbb{C})\\
	&=\dim\Hom_{\SL_2(F)}(ind_{\SL_2(F)}^{\Mp_{2}(E)}\mathbb{C},\tau)\\
	&=\dim\Hom_{\Mp_{2}(E) }(\mathcal{I}(0),\tau)\\
	&=\dim\Hom_{\Mp_{2}(E) }(R^{2,1}(\mathbf{1})\oplus R^{3,0}(\mathbf{1}),\tau)\\
	&=\dim \Hom_{PD^\times}(\pi(\mu,\mu^{-1} ),\mathbb{C})+ \dim \Hom_{\PGL_2(F) }(\pi(\mu,\mu^{-1}),\mathbb{C} ) \\
	&=\begin{cases}
	2,&\mbox{ if }\mu=\chi_F\circ N_{E/F};\\
	1,&\mbox{ if }\mu|_{F^\times}=\mathbf{1}\mbox{ and }\mu^2\neq\mathbf{1};\\
	0&\mbox{other cases}.
	\end{cases}
	\end{split}
	\end{equation*}
	Here we use the result for $\PGL(2)$-distinction problems that $\pi(\mu,\mu^{-1})$ is $\PGL_2(F)$-distinguished (resp. $PD^\times$-distinguished) if and only if the Langlands parameter $\mu+\mu^{-1}$ is conjugate-orthogonal (resp. $\mu$ factors through the norm map $N_{E/F}$). (See \cite[Theorem 2.5.3]{lu2016}.)
\item If $\tau $ is the even or odd Weil representation of $\Mp_2(E),$ set $V_a$ to be the $1$-dimensional quadratic space with a quadratic form $q(x)=ax^2,a\in E^\times/(E^\times )^2$ and $V'=\Res_{E/F}V_a$. Consider the following see-saw diagram  
\[\xymatrix{\tau &\Mp(W_{1,E})\ar@{-}[d]\ar@{-}[rd]& \Oo(V')\ar@{-}[d]\ar@{-}[ld]&R(\mathbf{1})  \\ \mathbb{C}&\Sp( W_1)& \Oo(V_a) } \]
where $\tau$ is a representation of $\Mp(W_{1,E})$ and $\mathbb{C}$ is the trivial representation of $\Sp(W_1)=\SL_{2}(F)$.
 If disc$V'\neq 1\in F^\times/(F^\times )^2$, then
the theta lift $R(\mathbf{1})$ to $\Oo(V')$  of the trivial representation of $\SL_2(F)$ is zero, so that $\Hom_{\SL_2(F)}(\omega^\pm_{\psi,\chi_a},\mathbb{C} )=0.$ If disc$V' =1\in F^\times/(F^\times )^2$, then $\Oo(V')=\Oo_{1,1}(F)$ and
\begin{equation*}
\begin{split}
\dim \Hom_{\SL_2(F)}(\omega^+_{\psi,\chi_a},\mathbb{C} )&=\dim \Hom_{\Oo(V_a)}(ind_{F^\times}^{\Oo_{1,1}(F)}|-|,\mathbf{1} )\\&=\begin{cases}
2,&\mbox{ if }N_{E/F}(a)=-\delta^2\in F^\times/(F^\times )^2; \\
0,&\mbox{ otherwise.}
\end{cases}
\end{split}
\end{equation*}
Similarly, if $N_{E/F}(a)=-\delta^2\in F^\times/(F^\times )^2 $, then we have the following identity $$\dim\Hom_{\SL_2(F)}(\omega^-_{\psi,\chi_a},\mathbb{C} )=\dim \Hom_{\Oo(V_a)}(ind_{F^\times}^{\Oo_{1,1}(F)}|-|,\det )=0. $$
	\end{enumerate}
This finishes the proof of Theorem \ref{mainthm}.
\end{proof}
\begin{rem}
	Although \cite[Proposition 7.2]{gan2014formal} is written in the sence of orthogonal group $\Oo_{2n+1}(F)$, it also works for the special orthogonal group $\SO_{2n+1}(F)$ due to the conservation relation.
\end{rem}

\section{Application to the Saito-Kurokawa lift}\label{sect.application}
 In  this section, we use the results of  metaplectic distinction problems to deal with the distinction problems for the split group $\PGSp_4\cong\SO_{3,2}$ and its pure inner form $\PGSp_{1,1}\cong\SO_{4,1}$ over a quadratic extension $E/F$.
\par
  Given a discrete series representation $\pi$ of $\PGL_2(E),$
one may consider the composition of theta lifts via
\[ \xymatrix{\Irr(PD_E^\times)\sqcup \Irr(\PGL_2(E))\ar[rr]^-{\theta_{\psi}}&&\Irr(\Mp(W_{1,E}))\ar[rr]^{\theta_{W_{1,E},V_{5,E},\psi}}&&\Irr(\PGSp_4(E)) } .\]
Then $\theta_{W_{1,E},V_{5,E},\psi}\circ \theta_{\psi}(\pi) $ and $\theta_{W_{1,E},V_{5,E},\psi}\circ \theta_{\psi}(\pi^{D_E} ) $  are called \textbf{Saito-Kurokawa lifts} of $\PGSp_4(E)$, where $\pi^{D_E}=JL(\pi)$ and $V_{5,E}=V^+_5\otimes_FE$. We will denote them by $SK(\pi)$ and $SK(\pi^{D_E}) $ respectively. If $\pi$ is an irreducible principal series representation, then $\pi^{D_E}$ does not exist.
\par
Given an irreducible square-integrable genuine representation $\tau=\Theta_{\psi}(\pi)$ of $\Mp(W_{1,E})$, 
the Saito-Kurokawa packet of $\PGSp_4(E)$ associated to $\pi$ has two elements
\[SK(\pi)= \theta_{V_{4,E},W_{2,E},\psi}(\pi\boxtimes\mathbb{C})\mbox{    and    }SK(\pi^{D_E})=\theta_{V_{D_E},W_{2,E},\psi}(\pi^{D_E}\boxtimes\mathbb{C} ) \]
where $V_{4,E}$ (resp. $V_{D_E}$) is the $4$-dimensional split (resp. non-split) quadratic space over $E$ with trivial discriminant, $\pi\boxtimes\mathbb{C}$ (resp. $\pi^{D_E}\boxtimes\mathbb{C}$) is an irreducible representation of $\mathrm{GSO}(V_{4,E})$ (resp. $\mathrm{GSO}(V_{D^E})$) and $W_{2,E}=W_2\otimes_FE$.
\begin{prop} Assume that $\pi$ is an irreducible representation of $\PGL_2(E).$ Then
	\begin{enumerate}[(i)]\label{saitokperiod}
		\item 	Given a square-integrable representation $\pi$ of $\PGL_2(E)$, then 
		\begin{enumerate}[(A)]
			\item 	$\dim\Hom_{\PGSp_4(F) }(SK(\pi^{D_E}),\mathbb{C})=0;$
			\item $\dim \Hom_{\PGSp_4(F)}(SK(\pi),\mathbb{C} )=\dim \Hom_{\PGU_2(D)}(SK(\pi),\mathbb{C} )
			=\begin{cases}
		2,&\mbox{if }\phi_{\pi} \mbox{ is conjugate-orthogonal};\\0,&\mbox{otherwise.}
		\end{cases} $
		\end{enumerate}
	\item If $\pi=\pi(\mu,\mu^{-1}) $ and $\tau=\pi_{\psi}(\mu)$, where $\mu\neq|-|_E^s$ and $s\in\{\pm 3/2,\pm1/2\},$ then \[\dim \Hom_{\PGSp_4(F) }(SK(\pi),\mathbb{C} )=\dim \Hom_{\PGSp_{1,1}(F)}(SK(\pi),\mathbb{C} ) =\dim \Hom_{\SL_2(F) }(\pi_{\psi}(\mu\cdot\chi_{\delta} ),\mathbb{C} ). \]
	\end{enumerate}
\end{prop}
\begin{proof}\begin{enumerate}[(i)]
		\item 
		\begin{enumerate}[(A)]
			\item It follows from \cite[Theorem 4.2.18]{lu2016}.
			\item 	Due to the see-saw diagram \eqref{doubleseesaw}, for $\tau=\theta_{\psi}(\pi)$, one has \[\dim\Hom_{\PGSp_4(F) }( SK(\pi),\mathbb{C})=\dim \Hom_{\Mp_2(E) }(\mathcal{I}(1),\tau )=\dim \Hom_{\PGSp_{1,1}(F) }(SK(\pi),\mathbb{C} ) . \]
			Then  the multiplicity $\dim\Hom_{\Mp_{2}(E)}(\mathcal{I}(1),\tau)  $ equals to 
			\begin{equation}\label{SL(2)}
			\dim \Hom_{\SL_2(F)}(\tau^\delta,\mathbb{C} )=2\dim \Hom_{\PGL_2(F) }(\pi,\mathbb{C} ),
			\end{equation} 
		where \eqref{SL(2)} holds	due to Theorem \ref{mainthm}. The desired identity follows from the results for $\PGL_2(F)$-distinction problems obtained in \cite[Theorem 2.5.2]{lu2016}, which  means that 
			$$\dim \Hom_{\PGSp_4(F)}(SK(\pi),\mathbb{C} )=\dim\Hom_{\PGSp_{1,1}(F)}(SK(\pi),\mathbb{C})= \begin{cases}
			2,&\mbox{if }\phi_{\pi} \mbox{ is conjugate-orthogonal};\\0,&\mbox{otherwise.}
			\end{cases} $$
		\end{enumerate}

		\item Thanks to \cite[Lemma 4.2]{gan2009restrictionsaito}, if $\mu$ is neither $|-|_E^{\pm 3/2 }$ nor $|-|_E^{\pm1/2}$, then the big theta lift $\Theta_{W_{1,E},V_{5,E},\psi}(\tau)$ is irreducible. Note that
		there are $2$ orbits for the double coset decomposition $\tilde{P}\backslash\Mp_4(F)/\Mp_2(E)$ in \eqref{Mp(4,F)-decomp}. Moreover, we have
		\[\Hom_{\Mp_{2}(E)}(\mathcal{I}(s_0)/\mathcal{I}_1(s_0),\tau )=\Hom_{\Mp_{2}(E)}(ind_{\tilde{Q}_1}^{\Mp_{2}(E)}(\chi_{\psi}|\det|_E^{s_0+1/2}),\tau )=0 \]
		if $s_0=1$ and $\mu\neq|-|_E^{3/2}$.
		 So $\tau$ does not occur on the boudary of $\mathcal{I}(1)$. Therefore, one has
		\begin{equation*}
		\begin{split}
		\dim \Hom_{\PGSp_4(F)}(SK(\pi),\mathbb{C} )&=	\dim \Hom_{\Mp_2(E)}(\mathcal{I}(1),\tau )\\
		&=\dim \Hom_{\SL_2(F)}(\tau^\delta,\mathbb{C})\\
&=		\dim \Hom_{\SL_2(F) }(\pi_{\psi}(\mu\cdot\chi_{\delta}),\mathbb{C}).
		\end{split}
		\end{equation*} 
	\end{enumerate}
Here we use the fact $\tau^\delta\cong\pi_{\psi}(\mu\cdot\chi_{\delta}).$
Together with Theorem \ref{mainthm}, we can obtain that
  if the character $\mu\neq|-|_E^s$ with $s\in\{\pm3/2,\pm1/2 \}$, then
\[\dim \Hom_{\PGSp_4(F)}(SK(\pi),\mathbb{C} )=\begin{cases}2,&\mbox{ if }\mu|_{E^1}=\mathbf{1};
\\1,&\mbox{ if }\mu|_{F^\times}=\mathbf{1}\mbox{ and }\mu^2\neq\mathbf{1};\\0,&\mbox{ otherwise.}
\end{cases} \]
Note that the see-saw diagram
\[\xymatrix{\PGSp_4(E)\ar@{-}[rd]\ar@{-}[d]&\Mp_4(F)\ar@{-}[ld]\ar@{-}[d] \\ \PGSp_{1,1}(F)&\Mp_2(E) } \]
implies the following 
	\[\dim \Hom_{\PGSp_{1,1}(F)}(SK(\pi),\mathbb{C} )=\dim \Hom_{\Mp_2(E)}(\mathcal{I}(1),\tau )=\dim \Hom_{\SL_2(F) }(\pi_{\psi}(\mu\cdot\chi_{\delta}),\mathbb{C} ). \]
	Hence we have completed the proof.
\end{proof}
\begin{rem}
There is a nontempered representation of $\PGSp_4(E)$ inside the Saito-Kurokawa packet, so it does not belong to the cases discussed in \cite{lu2018GSp(4)}, where the Prasad conjecture \cite{prasad2015arelative}  holds for the tempered representations of $\PGSp_4$.
\end{rem}
\section{On the Prasad conjecture}
In this section, we study the metaplectic distinction problem for higher dimension. Then we combine the Prasad conjecture to formulate a conjectural identity for the multiplicity
\[\dim\Hom_{\Sp_{2n}(F)}(\tau,\mathbb{C})  \]
where $\tau$ is a square-integrable representation of $\Mp_{2n}(E)$.
\subsection{Metaplectic distinction problems for $(\Mp_4(E),\Sp_4(F))$}\label{subsect.Mp(4)}
Using a similar idea, we have the following result for a tempered representation $\tau$ of $\Mp_4(E)$.
\begin{prop}\label{pgsp+pgu}
	Given a  tempered representation $\pi$ of $\PGSp_4(E)$
	and a representation $\tau^\delta=\theta_{\psi}(\pi)$ of $\Mp_4(E)$ with central character $\omega_{\tau^\delta}=\omega_\tau$ satisfying $\omega_\tau(-1,1)=\epsilon(1/2,\pi)/\gamma(-1,\psi),$ 
	then
	\begin{enumerate}[(i)]
		\item we have an identity	\begin{equation}\label{sumformforsp(4)}
		\dim \Hom_{\Sp_4(F)}(\tau,\mathbb{C} )=\dim \Hom_{\PGSp_4(F)}(\pi,\mathbb{C} )+\dim \Hom_{\PGSp_{1,1}(F)}(\pi,\mathbb{C} ),
		\end{equation}
		where $\PGSp_{1,1}$ is the unique pure inner form of $\PGSp_4$ defined over $F$;
		\item the multiplicity
		$\dim \Hom_{\Sp_4(F)}(\tau,\mathbb{C} ) $ is nonzero if and only if the Langlands parameter $\phi_\tau\otimes\chi_{\delta}=\phi_{\pi}$ is conjugate-orthogonal;
	\end{enumerate}
\end{prop}
\begin{proof}
	Here we give a general result for the pair $(\Mp_{2n}(E),\Sp_{2n}(F))$ and a tempered representation $\pi.$
	\par Recall that $V_{2n+1}^+$(resp. $V_{2n+1}^-$) is the $(2n+1)$-dimensional split (resp. non-split) quadratic space of trivial discriminant. Set $V_{2n+1,E}=V_{2n+1}^+\otimes_FE$. Assume that $\mathcal{I}(s)$ is the degenerate principal series representation of $\Mp_{4n}(F)$.
	Due to the following diagram 
	\begin{equation}\label{SO}
	\xymatrix{ \SO(V_{2n+1,E})\ar@{-}[d]\ar@{-}[rd] & \Mp_{4n}(F)\ar@{-}[d]\ar@{-}[ld]\ar@{-}[rd]& \SO(V_{2n+1,E})\ar@{-}[d] \\ \SO(V_{2n+1}^+) & \Mp_{2n}(E)\ar@{-}[ru] & \SO(V_{2n+1}^-) } \end{equation}
	 $\tau^\delta=\theta_{\psi}(\pi)$ is tempered due to Theorem\ref{1} and hence Proposition \ref{casselman} implies that $\tau^\delta$  does not occur on the boundary of $\mathcal{I}(0)$.   Then
	\[\dim \Hom_{\Sp_{2n}(F)}(\tau,\mathbb{C} )=\dim \Hom_{\Mp_{4n}(E)}(\mathcal{I}(0),\tau^\delta ). \]
	Thanks to \cite[Proposition 7.2]{gan2014formal}, the degenerate principal series $\mathcal{I}(0)$ is the direct sum 
	\begin{equation}\label{I(0)}
	\mathcal{I}(0)= R^{n+1,n}(\mathbf{1})\oplus R^{n+2,n-1}(\mathbf{1}) ,
	\end{equation}  where $R^{n+1,n}(\mathbf{1})$ (resp. $R^{n+2,n-1}(\mathbf{1})$) is the big theta lift
	of trivial representation from $\SO(V_{2n+1}^+)$ (resp. $\SO(V_{2n+1}^-)$) to $\Mp_{4n}(F)$. Then one can get
	\[\dim \Hom_{\Sp_{2n}(F)}(\tau,\mathbb{C} )=\dim \Hom_{\Mp_{2n}(E)}(R^{n+1,n}(\mathbf{1}),\tau^\delta )+\dim \Hom_{\Mp_{2n}(E)}(R^{n+2,n-1}(\mathbf{1}),\tau^\delta ) \]
	while the right hand side is equal to the sum $$\dim \Hom_{\SO(V_{2n+1}^+)}(\pi,\mathbb{C} )+\dim \Hom_{\SO(V_{2n+1}^-)}(\pi,\mathbb{C} )  $$  by the see-saw identities.
	Taking $n=2,$ we have $\SO(V_5^+)\cong \PGSp_4(F)$ and $\SO(V_5^-)\cong \PGSp_{1,1}(F)$. Then the desired identity \eqref{sumformforsp(4)} follows.
	\par
	The second part is the main result  in \cite[Theorem 4.2.18, Theorem 4.3.10]{lu2016}.
\end{proof}
  
\begin{rem}
	Let $\pi$ be a tempered representation  of $\SO(V_{2n+1,E})$ lying in an $L$-packet $\Pi_{\phi_{\pi}}$, where $\phi_{\pi}=\oplus_{i=1}^r m_i\phi_i$ and $\phi_i$ are irreducible. Then there is a Waldspurger's packet $Wd_{\psi}(\pi)$ of $\Mp_{2n}(E)$ associated to $\pi$ (see \cite{gan2012metaplectic}), which is given by
	\[Wd_{\psi}(\pi)=\{\tau\in\Irr(\Mp_{2n}(E) ) |\phi_\tau=\phi_{\pi} \}. \]
	 Given $\tau\in Wd_{\psi}(\pi)$ with an enhanced $L$-parameter $(\phi_{\tau},\eta_\tau)$, where $\eta_\tau$ is a character of $A_{\phi_{\tau}}\cong(\mathbb{Z}/2\mathbb{Z})^r$,  if $\eta_\tau(-1)=-\epsilon(1/2,\phi_{\tau})\epsilon(1/2,\phi_{\tau}\otimes\chi_{\delta})\langle-1,\delta\rangle^{1/2\sum_i\dim\phi_i}_E$, then \[\dim \Hom_{\Sp_{2n}(E) }(\tau,\mathbb{C} )=0. \]
\end{rem}

If $n=1$, it revisits the result in Theorem \ref{mainthm} that
\[\dim \Hom_{\SL_2(F)}(\tau,\mathbb{C} )=0\mbox{ if }\tau^\delta=\Theta_{\psi}(\pi^{D_E}), \]
without referring to the $\PGSp_4$-distinction problems over a quadratic field extension $E/F$ in \cite{lu2016}. 

\subsection{Relation with the Prasad conjecture}\label{subsect.prasad}
In order to introduce the Prasad conjecture, we need some recipes. Let $G$ be a quasi-split group defined over $F$.
Let $\rho$ be an irreducible representation of $G(F)$, i.e. $\rho\in \Irr(G(F))$. 
Assume the Langlands-Vogan conjectures
\cite[\S 9]{gan2011symplectic} for $G(F)$. 
Given a representation $\rho\in \Irr(G(F))$ with an enhanced $L$-parameter $$(\phi_\rho,\lambda),$$  where 
$\phi_\rho:WD_F\longrightarrow {}^LG=\hat{G}\rtimes W_F$ is a Langlands parameter  and   $\lambda:A_{\phi_{\rho}}\longrightarrow\mathbb{C}^\times$ is a character,
then $\phi_{\rho}|_{WD_E}$ is a Langlands parameter of $G(E)$.
Let $\phi_{\pi}:WD_E\longrightarrow \hat{G}\rtimes W_E$ be the Langlands parameter of $\pi$.
If $\phi_{\pi}=\phi_{\rho}|_{WD_E}$, then $\phi_{\rho}$ is called the parameter lift of $\phi_{\pi}$.
\par
Let $G^{op}$ be a quasi-split group over $F$ defined in \cite[\S9]{prasad2015arelative} satisfying $G^{op}(E)=G(E)$. Now we can give the statement of the Prasad conjecture, i.e. \cite[Conjecture 2]{prasad2015arelative}.
\begin{conj}[The Prasad conjecture]\label{prasadconjecture}
	Let $\pi$ be an irreducible admissible $G(F)$-distinguished representation of $G(E)$ with an enhanced $L$-parameter $(\phi_{\pi},\lambda)$,
	where $\pi$ lies in a generic L-packet $\Pi_{\phi_{\pi}}$ and $\lambda$ is a character of the component group $A_{\phi_{\pi}}$. Then
	\begin{equation}\label{prasadconj.}
	\sum_{\alpha}\dim\Hom_{G_\alpha}(\pi,\chi_G)=\sum_{i}m(\lambda,\tilde{\phi}_i)\deg\Phi(\tilde{\phi}_i)/d_0(\tilde{\phi}_i) 
	\end{equation}
	where
	\begin{itemize}
		\item $\alpha\in H^1(W_F,G)$ runs over all pure inner forms of $G$ satisfying $G_\alpha(E)=G(E);$
		\item $\tilde{\phi}_i\in \Hom(WD_F,{}^LG^{op})$ runs over all parameters of $^LG^{op}$ satisfying $\tilde{\phi}_i|_{WD_E}=\phi_\pi;$
		\item $m(\lambda,\tilde{\phi})=\dim\Hom_{A_{\tilde{\phi}}}(\mathbf{1},\lambda) $ is the multiplicity of the trivial representation contained in the restricted representation $\lambda|_{A_{\tilde{\phi}}}$;
		\item $\Phi:\Hom(WD_F,{}^LG^{op})\longrightarrow\Hom(WD_E,{}^LG^{op})$ is the base change map and $\deg\Phi$ is the degree;
		\item $d_0(\tilde{\phi})=|Coker\{A_{\tilde{\phi}}\longrightarrow A_{\phi_\pi}^{\Gal(E/F)} \}|$ is the size of the coker.
	\end{itemize}
\end{conj}
\begin{rem}
	If $\pi$ is a discrete series representation, then there is a formula for each individual dimension $\dim\Hom_{G_\alpha}(\pi,\chi_G)$. (See \cite[\S3.1]{lu2016}. )
\end{rem}
Here we  consider the Prasad conjecture  \cite{prasad2015arelative}
for the general spin group $G=\GSpin_{2n+1}$.
\par
Let $G=\GSpin_{2n+1}$. The  center $Z_G\cong\GL_1$ and the quotient group $G/Z_G$ is isomorphic to the special orthogonal group $\SO_{2n+1}.$ If $\Gal(E/F)=\langle\sigma\rangle,$ then $$G^{op}(F)=\{g\in G(E)|\sigma(g)=\lambda(g)^{-1}g \}$$
 and the quadratic character $\chi_G$ is  the character $\omega_{E/F}$ associated with the extension $E/F$ by Class Field Theory. 
\par Recall $V_{2n+1,E}=V_{2n+1}^+\otimes_FE$.
 Given a square-integrable representation $\pi$ of $\SO(V_{2n+1,E}),$ there is a representation $\Pi$ of $\GSpin(V_{2n+1,E})$ with trivial central character associated to $\pi$. Then
  we have
  \begin{equation}
  \begin{split}
  \dim \Hom_{\SO(V^+_{2n+1}) }(\pi,\mathbb{C} )&=\dim\Hom_{\GSpin(V^+_{2n+1})}(\Pi,\mathbb{C}) \\
  &=\dim\Hom_{\GSpin(V^+_{2n+1}) }(\Pi\otimes\chi_E,\omega_{E/F} ), 
  \end{split}
  \end{equation}
where $\chi_E$ is a character of $E^\times$ such that $\chi_E|_{F^\times}=\omega_{E/F},$ and the right hand side is related to the number of inequivalent lifts
of the Langlands parameter $\phi_{\Pi\otimes\chi_E }$ by  Conjecture \ref{prasadconjecture}.
\begin{conj} Let $E/F$ be a quadratic extension of  nonarchimedean local fields.  Suppose $\pi\in \Irr(\SO(V_{2n+1,E}) )$ is a square-integrable $\SO(V_{2n+1}^+)$-distinguished representation with an $L$-parameter $(\phi_{\pi},\eta),$ which determines a square-integrable $\GSpin(V_{2n+1}^+)$-distinguished representation $\Pi$ of $\GSpin(V_{2n+1,E})$ with trivial central character with associated $L$-parameter $(\phi_{\Pi},\lambda)$.
	Assume  that there exists a parameter $$\tilde{\phi}:WD_F\longrightarrow {}^LG^{op}$$ with component group $A_{\tilde{\phi}}$ such that $\tilde{\phi}|_{WD_E}=\phi_{\Pi\otimes\chi_E }$.
	Then for the square-integrable representation $\tau=\theta_{\psi}(\pi)$ of $\Mp_{2n}(E)$, 
	 one has
	\begin{equation}\label{conjidentity}
	\dim\Hom_{\Sp_{2n}(F)}(\tau^\delta,\mathbb{C} )=|\Irr(A_{\tilde{\phi}})|
	\end{equation}
	where $|\mathrm{Irr}(A_{\tilde{\phi}})	|$ denotes the number of irreducible representations of the finite group $A_{\tilde{\phi}}$.
\end{conj}
There is a conjectural identity for the square-integrable $G(F)$-distinguished representation $\Pi$ of $G(E)$
\begin{equation}\label{thesumforinnerform}
\dim\Hom_{\GSpin(V_{2n+1}^+)} (\Pi\otimes\chi_E,\omega_{E/F} )+ \dim\Hom_{\GSpin(V_{2n+1}^-)}(\Pi\otimes\chi_E,\omega_{E/F} )=|\Irr(A_{\tilde{\phi}})|
\end{equation}
where  $\tilde{\phi}|_{WD_E}=\phi_{\Pi\otimes\chi_E }=\phi_{\Pi}\otimes\chi_E$ is any lifted $L$-parameter of $G^{op}$. If \eqref{thesumforinnerform} holds, then the identity \eqref{conjidentity} holds due to the following
	\begin{equation*}
	\begin{split}
	\dim \Hom_{\Sp_{2n}(F) }(\tau^\delta,\mathbb{C} )&=\dim\Hom_{\GSpin(V^+_{2n+1}) }(\Pi\otimes\chi_E,\omega_{E/F} )+\dim\Hom_{\GSpin(V^-_{2n+1}) }(\Pi\otimes\chi_E,\omega_{E/F} )\\
	&=\dim\Hom_{\SO(V_{2n+1}^+)}(\pi,\mathbb{C})+\dim\Hom_{\SO(V_{2n+1}^-)}(\pi,\mathbb{C}) \\
	&=\dim\Hom_{\Mp_{2n}(E)}(R^{n+1,n}(\mathbf{1}),\tau )+\dim\Hom_{\Mp_{2n}(E)}(R^{n+2,n-1}(\mathbf{1}),\tau)\\
	&=\dim\Hom_{\Mp_{2n}(E)}(I(0),\tau), 
	\end{split}
	\end{equation*}
	where the last equality holds due to \eqref{I(0)}. 
	 \begin{rem}
		Raphael Beuzart-Plessis \cite{beuzart2017distinguished} proved that the multiplicity
		\[\dim\Hom_{G_\alpha(F)}(\pi,\chi_G) \]
		is independent of the choice of the inner form $G_\alpha$,
		where $\pi$ is a stable square-integrable representation of $G(E)$ and $G_\alpha$ is the inner form of $G$ satisfying $G_\alpha(E)=G(E)$. However, the conjectural identity \eqref{thesumforinnerform} involves arbitrary irreducible $G(F)$-distinguished representation which may not be stable. 
	\end{rem}
In fact, Conjecture \ref{prasadconjecture} has been proved for 
	 $\PGL_2$ and
	 $\PGSp_4$ if $\pi$ is a tempered representation of $\PGSp_4(E)$ in  \cite{lu2016,lu2018GSp(4)}.
	 
  Proposition \ref{pgsp+pgu} holds for   a tempered representation $\tau$ of $\Mp_4(E)$. However, we do not know  whether \eqref{sumformforsp(4)} holds or not if $\tau$ is a non-tempered representation of $\Mp_4(E)$.

Notice that $\widehat{\GSpin_{2n+1}}=\GSp_{2n}(\mathbb{C}).$ Given an irreducible parameter $$\phi_{\Pi\otimes\chi_E}:WD_E\longrightarrow \GSp_{2n}(\mathbb{C}),$$ there exists at most one lift
\[\phi:WD_F\longrightarrow \GSp_{2n}(\mathbb{C})\rtimes\sigma \]
such that $\phi|_{WD_E}=\phi_{\Pi}\otimes\chi_E,$ where the action of $\sigma$ on $\GSp_{2n}(\mathbb{C})$ is given by
\[\sigma(g)=\lambda(g)^{-1}g \]
for $g\in\GSp_{2n}(\mathbb{C})$.
Then \eqref{thesumforinnerform} implies that
 $$\dim\Hom_{\Sp_{2n}(F) }(\tau^\delta,\mathbb{C} )=\big|\mathrm{Irr}(A_\phi) \big|$$ if $\Pi$ is a square-integrable $\GSpin_{2n+1}(F)$-distinguished representation of $\GSpin_{2n+1}(E)$ with trivial central character.

It is believable that the pair $(\Mp_{2n}(E),\Sp_{2n}(F))$ is not a Gelfand Pair, i.e. for arbitrary $n$, there exists a representation $\tau\in \Irr(\Mp_{2n}(E))$ such that 
\[\dim\Hom_{\Sp_{2n}(F)}(\tau,\mathbb{C})>1.  \]
Indeed, we have the following.
\begin{coro}
	If $n$ is either $1$ or $2,$ then $(\Mp_{2n}(E),\Sp_{2n}(F))$ is not  a Gelfand pair.
\end{coro}
\begin{proof}
	It follows directly from Theorem \ref{mainthm} and Proposition \ref{pgsp+pgu}.
\end{proof}
\bibliographystyle{abbrv}
\bibliography{Mp(2)}

\begin{thebibliography}{10}

\bibitem{ban2013langlands}
D.~Ban and C.~Jantzen.
\newblock The {L}anglands quotient theorem for finite central extensions of
  p-adic groups.
\newblock {\em Glas. Mat. Ser. III}, 48(68):313--334, 2013.

\bibitem{beuzart2017distinguished}
R.~Beuzart-Plessis.
\newblock On distinguished square-integrable representations for {G}alois pairs
  and a conjecture of {P}rasad.
\newblock {\em Invent. Math.}, 241(1):437--521, 2018.

\bibitem{lapid2012unitary}
B.~Feigon, E.~Lapid, and O.~Offen.
\newblock On representations distinguished by unitary groups.
\newblock {\em Publ. Math. Inst. Hautes \'Etudes Sci.}, 115:185--323, 2012.

\bibitem{flicker1991ondist}
Y.~Z. Flicker.
\newblock On distinguished representations.
\newblock {\em J. Reine Angew. Math.}, 418:139--172, 1991.

\bibitem{gan2011shimura}
W.~T. Gan.
\newblock The {S}himura correspondence a la {W}aldspurger.
\newblock {\em Lecture notes}, 2011.

\bibitem{gan2011symplectic}
W.~T. Gan, B.~H. Gross, and D.~Prasad.
\newblock Symplectic local root numbers, central critical {L}-values, and
  restriction problems in the representation theory of classical groups.
\newblock {\em Ast{\'e}risque}, 346:1--109, 2011.

\bibitem{gan2009restrictionsaito}
W.~T. Gan and N.~Gurevich.
\newblock Restrictions of {S}aito-{K}urokawa representations.
\newblock In {\em Automorphic forms and {$L$}-functions {I}. {G}lobal aspects},
  volume 488 of {\em Contemp. Math.}, pages 95--124. Amer. Math. Soc.,
  Providence, RI, 2009.
\newblock With an appendix by Gordan Savin.

\bibitem{gan2014formal}
W.~T. Gan and A.~Ichino.
\newblock Formal degrees and local theta correspondence.
\newblock {\em Invent. Math.}, 195(3):509--672, 2014.

\bibitem{gan2012metaplectic}
W.~T. Gan and G.~Savin.
\newblock Representations of metaplectic groups {I}: epsilon dichotomy and
  local {L}anglands correspondence.
\newblock {\em Compos. Math.}, 148(6):1655--1694, 2012.

\bibitem{gan2010shalika}
W.~T. Gan and S.~Takeda.
\newblock On {S}halika periods and a theorem of {J}acquet-{M}artin.
\newblock {\em Amer. J. Math.}, 132(2):475--528, 2010.

\bibitem{gan2014howe}
W.~T. Gan and S.~Takeda.
\newblock On the {H}owe duality conjecture in classical theta correspondence.
\newblock In {\em Advances in the theory of automorphic forms and their
  {$L$}-functions}, volume 664 of {\em Contemp. Math.}, pages 105--117. Amer.
  Math. Soc., Providence, RI, 2016.

\bibitem{gan2014proof}
W.~T. Gan and S.~Takeda.
\newblock A proof of the {H}owe duality conjecture.
\newblock {\em J. Amer. Math. Soc.}, 29(2):473--493, 2016.

\bibitem{kudla1996notes}
S.~S. Kudla.
\newblock Notes on the local theta correspondence.
\newblock {\em unpublished notes, available online}, 1996.

\bibitem{Kudla1992}
S.~S. Kudla and S.~Rallis.
\newblock Ramified degenerate principal series representations for {$\rm
  Sp$}(n).
\newblock {\em Israel Journal of Mathematics}, 78(2):209--256, 1992.

\bibitem{kudla2005first}
S.~S. Kudla and S.~Rallis.
\newblock On first occurrence in the local theta correspondence.
\newblock {\em Automorphic representations, {$L$}-functions and applications:
  progress and prospects}, 11:273--308, 2005.

\bibitem{lu2016}
H.~Lu.
\newblock {\em $\rm{GSp(4)}$ period problems over a quadratic field extension}.
\newblock PhD thesis, National University of Singapore, 2017.

\bibitem{lu2018GSp(4)}
H.~Lu.
\newblock The {P}rasad conjectures for $\rm {GSp_4}$ and $\mathrm{PGSp_4}$.
\newblock {\em arXiv preprint arXiv:1802.10336v1}, 2018.

\bibitem{matringe2009distinction}
N.~Matringe.
\newblock Distinguished generic representations of $\rm{GL(n)}$ over p-adic
  fields.
\newblock {\em Int. Math. Res. Not. IMRN}, 2011(1):74--95, 2011.

\bibitem{Dipendra1992invariant}
D.~Prasad.
\newblock Invariant forms for representations of {${\rm GL}_2$} over a local
  field.
\newblock {\em Amer. J. Math.}, 114(6):1317--1363, 1992.

\bibitem{prasad1996some}
D.~Prasad.
\newblock Some applications of seesaw duality to branching laws.
\newblock {\em Mathematische Annalen}, 304(1):1--20, 1996.

\bibitem{prasad2015arelative}
D.~Prasad.
\newblock A 'relative' local {L}anglands correspondence.
\newblock {\em arXiv preprint arXiv:1512.04347}, 2015.

\bibitem{sakellaridis2012periods}
Y.~Sakellaridis and A.~Venkatesh.
\newblock Periods and harmonic analysis on spherical varieties.
\newblock {\em Ast\'{e}risque}, 396, 2017.

\bibitem{schmidt2005saito}
R.~Schmidt.
\newblock The {S}aito-{K}urokawa lifting and functoriality.
\newblock {\em Amer. J. Math.}, 127(1):209--240, 2005.

\bibitem{sun2012conservation}
B.~Sun and C.-B. Zhu.
\newblock Conservation relations for local theta correspondence.
\newblock {\em J. Amer. Math. Soc.}, 28(4):939--983, 2015.

\bibitem{waldspurger1990demonstration}
J.-L. Waldspurger.
\newblock Demonstration d'une conjecture de dualit\'{e} de {H}owe dans le cas
  p-adiques, p$\ne$ 2, in {F}estschrift in honor of {I}. {P}iatetski-{S}hapiro.
\newblock In {\em Israel Math. Conf. Proc.}, volume~2, pages 267--324.
  Weizmann, Jerusalem, 1990.

\bibitem{waldspurger1991correspondances}
J.-L. Waldspurger.
\newblock Correspondances de {S}himura et quaternions.
\newblock {\em Forum Math.}, 3(3):219--308, 1991.

\end{thebibliography}
\end{document}